\documentclass[12pt]{article}
\input epsf
\usepackage{epsfig}
\usepackage{pict2e}
\usepackage{color}
\usepackage[percent]{overpic}
\usepackage{amsmath}
\usepackage{amsthm}
\usepackage{amssymb}
\usepackage[percent]{overpic}
\def\g{\widetilde{g}}
\def\F{\widetilde{F}}
\def\h{\widetilde{h}}
\def\tgamma{\widetilde{\gamma}}
\def\p{\widetilde{p}}
\def\re{\operatorname{Re}}
\def\im{\operatorname{Im}}
\def\Ima{\operatorname{Im}}
\def\N{\mathbb{N}}
\def\Z{\mathbb{Z}}
\def\R{\mathbb{R}}
\def\C{\mathbb{C}}
\def\bC{\overline{\C}}

\newtheorem*{theorema}{Theorem}
\newtheorem*{theoremf}{Fatou's Theorem}
\newtheorem*{theoremh}{Hurwitz's Theorem}
\newtheorem{lemma}{Lemma}
\theoremstyle{remark}
\newtheorem*{comment}{Comment}
\newtheorem{remark}{Remark}
\begin{document}
\title{Green's function and anti-holomorphic dynamics on a torus}
\author{Walter Bergweiler and Alexandre Eremenko\thanks{Supported by NSF grant DMS-1361836.}}
\date{}
\maketitle
\begin{abstract}
We give a new, simple proof of the fact recently discovered by
C.-S.~Lin and C.-L.~Wang that the Green function of a
torus has either three or five critical
points, depending on the modulus of the torus.
The proof uses anti-holomorphic dynamics.
As a byproduct we find a one-parametric family of anti-holomorphic dynamical
systems for which the parameter space consists
only of hyperbolic components
and analytic curves separating them.
\end{abstract}

\section{Introduction}
Green's function on a torus $T$ is defined as a solution of the equation
$$\Delta G=-\delta+\frac{1}{|T|},$$
normalized so that 
$$\int_TG=0.$$
Here $\delta$ is the delta-function, and $|T|$
is the area of $T$ with respect to a flat metric.

We write the torus $T$ as $T=\C / \Lambda$ with  
a lattice 
\[
\Lambda=\{ m\omega_1 +n\omega_2\colon m,n\in\Z\},
\]
where $\tau=\omega_2/\omega_1$ satisfies $\im\tau>0$. 
Recently C.-S.\ Lin and C.-L.\ Wang~\cite{Lin}
discovered that Green's function has either three
or five critical points, depending on $\tau$. It is surprising that this
simple fact was not known until 2010.
In \cite{Lin1,Lin2}
they study the corresponding partition of the $\tau$-half-plane.
Their proofs are 
long and indirect, using advanced non-linear PDE theory.
Our paper is motivated by the desire to give a simple proof of their result
that Green's function has either three or five critical points and to 
give a criterion for $\tau$ distinguishing which case occurs.

We have (see~\cite{Lin})
\[
G(z)=-\frac{1}{2\pi}\log|\theta_1(z)|+\frac{(\Ima z)^2}{2\Ima\tau}+C(\tau),
\]
where $\theta_1$ is the first theta-function. 
Here and in the following we use the notation of elliptic functions
as given in~\cite{Ahlfors,HC}.
We note that the notation in \cite{A,Weierstrass,WW} is different, see
the remark following the theorem below.

Critical points of $G$ are solutions of the equation
\begin{equation}\label{1}
\zeta(z)+az+b\overline{z}=0,
\end{equation}
where the constants $a$ and $b$ are uniquely defined by the condition
that the left hand side is $\Lambda$-periodic. With $\zeta(z+\omega_j)=
\zeta(z)+\eta_j$ for $j=1,2$ we thus have 
\[
\eta_1+a \omega_1+b\overline{\omega_1}=0
\quad\text{and}\quad
\eta_2+a \omega_2+b\overline{\omega_2}=0.
\]
With the Legendre relation $\eta_1\omega_2-\eta_2\omega_1=2\pi i$ we obtain
\begin{equation}\label{ab}
b=-\frac{\pi}{|\omega_1|^2 \im \tau}
\quad\text{and}\quad
a=-\frac{b\overline{\omega_1}}{\omega_1}-\frac{\eta_1}{\omega_1}
=\frac{\pi}{\omega_1^2 \im \tau}-\frac{\eta_1}{\omega_1}.
\end{equation}
So the problem is to determine the number of solutions of~\eqref{1} where $a$ and $b$ are
given by~\eqref{ab}.
\begin{theorema}
The equation~\eqref{1} has three solutions in $T$ if
$e_j\omega_1^2+\eta_1\omega_1=0$ or
\begin{equation}\label{criterion2}
\im\!\left(\frac{2\pi i}{e_j\omega_1^2+\eta_1\omega_1} - \tau\right)\geq 0
\end{equation}
for some $j\in\{1,2,3\}$ and it has  five solutions otherwise.
\end{theorema}
Here, as usual, $e_1=\wp(\omega_1/2)$, 
$e_2=\wp((\omega_1+\omega_2)/2)$ and $e_3=\wp(\omega_2/2)$.
An elementary computation shows that the condition
in the theorem
is equivalent to
\[
\min_{1\leq j\leq 3}\left|\frac{e_j\omega_1^2+\eta_1\omega_1}{\pi}\im\tau-1\right|\leq 1
\]

We note that $e_j\omega_1^2$ and $\eta_1\omega_1$ depend only on $\tau=\omega_2/\omega_1$.
We may restrict to the case that $\omega_1=1$ so that $\tau=\omega_2$.
Then~\eqref{criterion2} simplifies to
\[
\im\!\left(\frac{2\pi i}{e_j+\eta_1} - \tau\right)\geq 0.
\]

As mentioned, a different notation for elliptic functions is used in \cite{A,Weierstrass,WW}.
There the periods are denoted by $2\omega_j$ and the definition of $\eta_1$ also 
differs by a factor~$2$.
Thus in that terminology~\eqref{criterion2} takes the form
\[
\im\!\left(\frac{\pi i}{2(e_j\omega_1^2+\eta_1\omega_1)} - \tau\right)\geq 0.
\]

Figure~\ref{fig1} shows (in gray) the regions in the $\tau$-plane where Green's function
has $5$ critical points; that is, the set of $\tau$-values where~\eqref{criterion2}
fails for all~$j$. The range shown is $|\re \tau|\leq 1$ and $0.15\leq \im \tau\leq 2.15$.
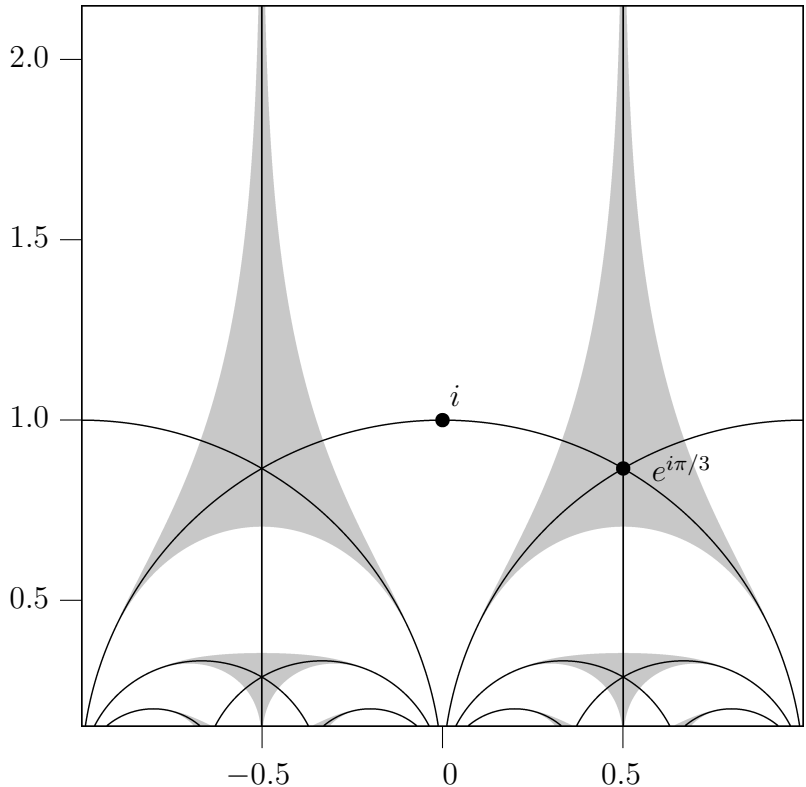
\begin{figure}[htb]
\begin{center}
\begin{overpic}[width=0.70\textwidth]{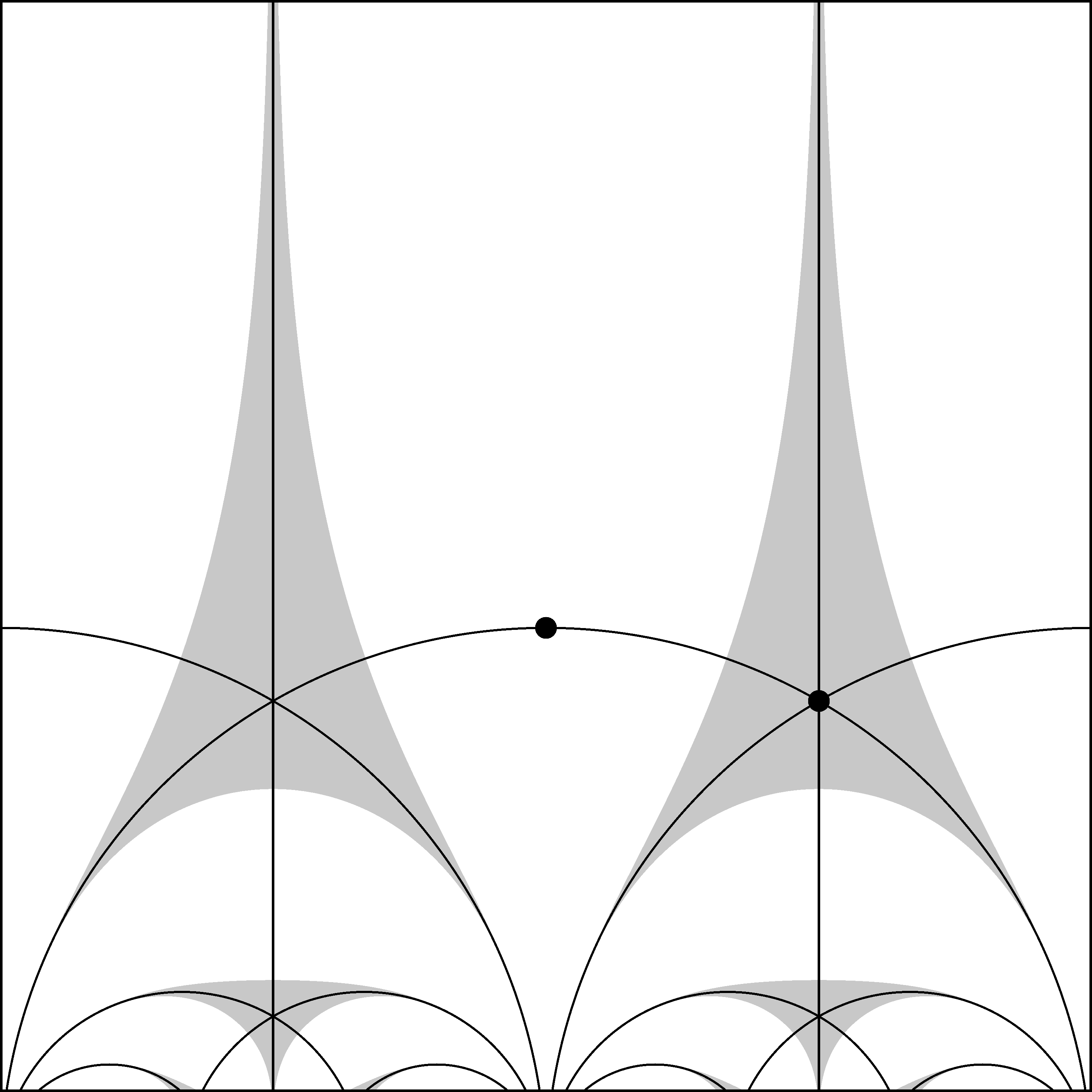}
 \put (25,0) {\line(0,-1){3}}
 \put (50,0) {\line(0,-1){3}}
 \put (75,0) {\line(0,-1){3}}
 \put (20,-8) {$-0.5$}
 \put (50,-8) {$0$}
 \put (72,-8) {$0.5$}
 \put (0,17.5) {\line(-1,0){3}}
 \put (0,42.5) {\line(-1,0){3}}
 \put (0,67.5) {\line(-1,0){3}}
 \put (0,92.5) {\line(-1,0){3}}
 \put (-10,16.5) {$0.5$}
 \put (-10,41.5) {$1.0$}
 \put (-10,66.5) {$1.5$}
 \put (-10,91.5) {$2.0$}
 \put (51,44.5) {$i$}
 \put (79,34.0) {$e^{i\pi/3}$}
\end{overpic}
\vspace*{0.7cm}
\caption{The regions given by~\eqref{criterion2} in the $\tau$-plane.}
\label{fig1}
\end{center}
\end{figure}

The standard fundamental domain consisting of those $\tau$ which satisfy
the inequalities $\im \tau>0$, $-1/2<\re \tau\leq 1/2$ and
$|\tau|\geq  1$, with $|\tau|=1$ only if $\re \tau\geq 0$, is in the upper middle of the picture.
Its images under the modular group are also shown.

Our proof is based on Fatou's theorem from complex
dynamics. Originally Fatou's theorem was proved to estimate from above
the number of attracting cycles of a rational function. Then it was extended
to more general classes of functions. The most surprising fact is that Fatou's
theorem can be used sometimes to estimate the number of solutions of equations in
settings where dynamics is not present. This was first noticed in
\cite{E}; the contents of this unpublished preprint is reproduced in
\cite{H,BEL}. The paper \cite{BE} shows that Fatou's theorem can be used
to prove under some circumstances the existence of critical
points of a meromorphic function.
In the papers \cite{KS,KN1,KN2,KL} Fatou's theorem was used
to obtain upper estimates of the numbers of solutions of equations
of the form 
$$z=r(\overline{z}),$$
with a meromorphic function $r$; this permitted to prove a 
conjecture in astronomy \cite{KN1,KN2}.
In the recent work \cite{M}, a topological classification of quadrature domains
is obtained with a method based on Fatou's theorem.

As mentioned, Fatou \cite[\S 30]{Fatou1920a} stated his result originally only for 
rational functions. As pointed out for example in~\cite[Theorem~7]{B}, the proof extends
to functions meromorphic in~$\C$.

To state the version of  Fatou's theorem that we need, we recall some
definitions.
Let $S$ and $T$ be Riemann surfaces and let $f\colon S\to T$ be holomorphic.
A point $c\in S$ is called a {\em critical} if $f'(c)=0$;
this condition does not depend on the local coordinates.

We say that a curve $\gamma\colon [0,1)\to S$ {\em escapes} if for every compact subset
$K$ in $S$ there exists $t_0\in(0,1)$ such that $\gamma(t)\not\in K$ for
$t\in[t_0,1)$. A curve $\gamma\colon [0,1)\to S$ is called an {\em asymptotic curve}
of $f$ if $\gamma$ escapes and the limit $\lim_{t\to 1} f(\gamma(t))$ exists and
is contained in~$T$. This limit is called an {\em asymptotic value} of $f$.
The following result is due to Hurwitz~\cite{H0}.

\begin{theoremh} 
If a holomorphic map between two Riemann surfaces
has no critical points and no asymptotic curves, then it is a covering.
\end{theoremh}

This result can be deduced from the fact that covering maps are characterized by the path-lifting property:
for every curve $\gamma\colon [0,1]\to T$ and every $z_0\in f^{-1}(\gamma(0))$
there exists a unique curve $\gamma^*\colon [0,1]\to S$ such that
$\gamma=f\circ\gamma^*$; see, for example, \cite[Theorem 9.1]{CI}, where the term
``complete covering'' is also used.
Note also that critical points and asymptotic curves correspond to singularities of
the inverse function; see~\cite[\S XI.1]{NevanlinnaEAF}.

A Riemann surface $S$ is called {\em hyperbolic} if its universal covering
is the unit disk \cite{CI}. A hyperbolic Riemann surface is equipped with the
hyperbolic metric. We denote by $\lambda_S(z)|dz|$ 
the length element of the hyperbolic metric in $S$.
The invariant form of the Schwarz lemma (see, for example,
\cite[Proposition I.2.8 (a)]{Thurston}, \cite[Theorem~2.11]{Milnor2} or \cite[Theorem 10.5]{BM})
 says that a holomorphic map $f\colon S\to T$ between
hyperbolic Riemann surfaces $S$ and $T$ satisfies
$\lambda_T(f(z))|f'(z)| \leq \lambda_S(z)$, with strict inequality
for all $z$, unless $f$ is a covering, in which case the equality
holds for all $z$.

For a holomorphic map $f\colon S\to S$
a point $z_0\in S$ is called {\em fixed}
if $f(z_0)=z_0$, and for such a point $f'(z_0)$ is called the {\em multiplier}.
It is easy to verify that the multiplier does not depend on
the local coordinate. A fixed point is called
{\em attracting}, \emph{neutral} or \emph{repelling}
depending on whether the modulus of its multiplier is less than,
equal to or greater than~$1$,
respectively.

\begin{theoremf}
Let $S$ be a hyperbolic Riemann surface and let $f\colon S\to S$
be a holomorphic map with an attracting fixed point
$z_0\in S$. Then f has a critical point or an asymptotic
curve in~$S$. Moreover, $f^n\to z_0$ locally uniformly in~$S$.
\end{theoremf}

\begin{proof} If $f$ has no critical points and no asymptotic curves,
then $f$ is a covering. 
The Schwarz lemma yields $\lambda_S(z_0)=\lambda_S(f(z_0))|f'(z_0)|=
\lambda_S(z_0)|f'(z_0)|$ and thus $|f'(z_0)|=1$, 
contradicting the assumption that
$|f'(z_0)|<1$. The second statement also follows from the Schwarz lemma.
\end{proof}

In applications to holomorphic dynamics, $S$ is the
immediate attraction basin of an attracting fixed
point $z_0$; that is, the component of the Fatou set which
contains $z_0$. It then follows that
the immediate attracting basin contains a critical point
or an asymptotic curve.


\section{Proof of the Theorem}

To prove the theorem we rewrite~(\ref{1}) as a fixed point equation
\begin{equation}\label{2}
z=-\frac{1}{b}\left(\overline{\zeta(z)}+\overline{az}\right)=:g(z),
\end{equation}
so $g$ is an {\em anti}-meromorphic function in the plane.
As the left hand side of~\eqref{1} is $\Lambda$-periodic, we conclude from~\eqref{2} that
\begin{equation}\label{3}
g(z+\omega)=g(z)+\omega,\quad\omega\in\Lambda.
\end{equation}

The function $g$ does not map the plane into itself because it has poles,
so the equation~(\ref{3}) does not permit to define a map of the torus $T$ into itself.
To remedy this, we consider the set $P_0$ of poles of~$g$.
For $n\geq 1$ we define inductively $P_n=g^{-1}(P_{n-1})$. Then
all iterates of $g$ are defined on the set $\C\backslash P_\infty$,
where
$P_\infty=\bigcup_{n=0}^\infty P_n$. Let $J$ be the closure of $P_\infty$.
Then the iterates of $g$ form a normal family in $\C\backslash J$, and this
set is completely invariant under $g$.
Thus -- with an obvious extension
of these concepts from holomorphic functions
to anti-holomorphic ones -- we call $J$ the \emph{Julia set} and
its complement $F=\C\backslash J$ the \emph{Fatou set} of~$g$.
The Fatou set is thus the maximal open subset of the plane such that
$g(F)\subset F$.
Evidently, $F$ is $\Lambda$-invariant,
so the map $g\colon F\to F$ descends to a map which is defined on an open
subset of the torus $T$ and maps this open subset to itself.

Let $\pi\colon \C\to\C/\Lambda=T$ be the projection map,
$\F=\pi(F)$, and $\g\colon \F\to \F$ the induced map which satisfies
$\g\circ\pi=\pi\circ g.$ The Riemann surface $\F$ is a subset
of the torus, and it is hyperbolic because its complement is infinite.

We apply the terminology
attracting, neutral and repelling also to fixed points $z_0$ of
anti-holomorphic maps $f$,
considering $\overline{\partial}f(z_0)$ as the multiplier.
Notice that
\begin{equation}\label{dg}
\overline{\overline{\partial} g(z)}=-\frac{1}{b}\left(\zeta'(z)+a\right)
=\frac{1}{b}(\wp(z)-a).
\end{equation}
To obtain holomorphic dynamics instead of the anti-holomorphic one, we consider the
second iterates $h=g^2$ and $\h=\g^2$. Then we have
\begin{equation}\label{semi}
\h\circ\pi=\pi\circ h.
\end{equation}
Images of fixed points of $g$ and $h$ in $F$ under $\pi$ are fixed points of
$\g$ and $\h$, respectively.
By the chain rule, we have
\begin{equation}\label{dg2}
h'=((\overline{\partial}g)\circ g)\cdot\partial \overline{g}
=((\overline{\partial}g)\circ g)\cdot\overline{\overline{\partial} g}.
\end{equation}
For a fixed point $z_0$ of $g$ we thus obtain
$h'(z_0)=|\overline{\partial}g(z_0)|^2$.
Even though we will not need this fact, we observe that the multiplier
of $h$ at a fixed point of $g$
is always a non-negative real number.

In order to apply Fatou's theorem to $\h\colon \F\to \F$,
we have to consider the critical points and asymptotic curves
of~$\h$.
\begin{lemma}
The map $\h\colon \F\to \F$ has no asymptotic curves.
\end{lemma}
\begin{proof}
We prove this by contradiction.
Let $\tgamma$ be an asymptotic curve of $\h$ in~$\F$. Let
$\gamma$ be some lifting of $\tgamma$ in $F$;
that is $\pi\circ\gamma=\tgamma$.
As $\tgamma$ escapes from~$\F$, the curve $\gamma$ escapes from~$F$.
By assumption, we have $\h(\tgamma(t))\to \p$ as $t\to 1,$ for some $\p\in\F$, so
$\h(\pi(\gamma(t)))\to\p$ as $t\to 1$.
In view of (\ref{semi}) this yields $\pi(h(\gamma(t)))\to\p$ as $t\to 1$. 
Because $\pi$ is a covering we conclude that
\begin{equation}
\label{77}
h(\gamma(t))\to p\quad\mbox{for some}\  p\in F,
\end{equation}
with $\pi(p)=\p$. Thus $\gamma$ is an asymptotic curve of $h$ in $F$.

Suppose that $\gamma$ has a limit point $q\in\C$ as $t\to 1$; that is,
there exists a sequence $(t_j)$ tending to $1$ such that $\gamma(t_j)\to q$.
Then $q\in \partial F\subset J$.
If $q$ is neither a pole of $g$ nor a
preimage of such a pole under $g$, then $h$ is holomorphic at $q$,
with $h(q)\in J$ by the complete invariance of the Julia set.
On the other hand, we have $h(q)=\lim_{j\to\infty}h(\gamma(t_j))= p\in F$,
which is a contradiction.
If $q$ is the preimage of a pole of~$g$, then $q$ is a pole of~$h$.
Thus $h(q)=\infty$, contradicting again $h(q)=p$.

This shows that the finite limit points of $\gamma$ are poles of $g$.
As the poles form a discrete subset of $\C$ we actually see that 
if $\gamma$ has a finite limit point $q$, then $\gamma(t)\to q$ 
as $t\to 1$ and thus $\sigma(t):=g(\gamma(t))\to\infty$ as $t\to 1$.
On the other hand, $g(\sigma(t))=h(\gamma(t))\to p$ as $t\to 1$.
Let $E$ be a $\Lambda$-invariant set consisting of disjoint disks around
the poles of $g$.  It follows from~\eqref{3} that 
\[
g(z)\to\infty\quad\mbox{as}\  z\to\infty,\  z\in\C\backslash E.
\]
This is a contradiction to $\sigma(t)\to\infty$ and $g(\sigma(t))\to p$ as $t\to 1$.
We have thus shown that $\gamma$ has no finite limit points, meaning that
$\gamma(t)\to\infty$ as $t\to 1$.

%
In order to show that this is impossible we note again that the 
singularities of $h$ are the poles of $g$ and their preimages under~$g$.
Since these preimages accumulate only at the poles of $g$ (and at $\infty$) there 
exists a $\Lambda$-invariant set $E'\subset\C$ consisting of disjoint disks
which contains all singularities of $h$.
Since $h$ satisfies
\[
h(z+\omega)=h(z)+\omega,\quad\omega\in\Lambda,
\]
by (\ref{3}), it follows 
that
\[
h(z)\to\infty\quad\mbox{as}\  z\to\infty,\  z\in\C\backslash E'.
\]
As before this is incompatible with $\gamma(t)\to\infty$ and $h(\gamma(t))\to p$
as $t\to 1$.
This completes the proof of the lemma.
\end{proof}

We consider the equivalence relation on $\F$ defined by $z\sim z'$ if
$\h^n(z)=\h^m(z')$ for some
non-negative integers $m$ and~$n$. The equivalence classes are
called the {\em grand orbits}.
We note that if the sequence $(\h^n(z))$ converges for some $z\in \F$,
then
for all $z'$ in the grand orbit of $z$ the sequence $(\h^n(z'))$
converges to the same limit.
We call this the limit of the grand orbit.

\begin{lemma}
The set of critical points of $\h$ belongs to at most $4$
grand orbits under~$\h$.
\end{lemma}

\begin{proof}
By~\eqref{dg}, the equation $\overline{\partial}g(z)=0$ is equivalent to
$\wp(z)=a$,
and has two solutions modulo $\Lambda$ which define points
$c$ and $-c$ on $T$. Here and in other similar places
we denote by $-c$ the point which corresponds to $c$ by the conformal
involution of the torus.
If any of the points $c$ and $-c$ is in $\F$, then it is a zero
of $\h'$.
By~\eqref{dg2} the other zeros of $\h'$ are $\g^{-1}(\pm c) =\pm\g^{-1}(c)$.
Even though this is an infinite
set on~$T$, the critical points of $\h'$ are thus contained
in at most $4$ grand orbits represented by
$c$, $-c$, $\g(c)$ and $-\g(c)$.
\end{proof}

\begin{lemma}
Let $z_0$ be a fixed point of~$\g$. If $\h^n(z)\to z_0$
as $n\to\infty$ for some~$z$,
then $\h^n(\g(z))\to z_0$.
\end{lemma}
\begin{proof}
We have $\h^n\circ \g=\g\circ \h^n$ for all $n\in\N$.
Since $\h^n(z)\to z_0$ this yields
$\h^n(\g(z_0))=\g(\h^n(z))\to \g(z_0)=z_0$.
\end{proof}

Applying Fatou's theorem to the map $\h\colon \F\to\F$,
we deduce that $\h$ and thus $\g$ has at most two attracting fixed points.
If $\F$ is disconnected, we apply Fatou's theorem separately to each component
that contains an attracting fixed point. As the $\pi$-image
of an attracting fixed point of $g$ is
an attracting fixed point of $\g$ we obtain
\begin{lemma} \label{lemma4}
The function $g$ has at most two attracting
fixed points, modulo~$\Lambda$.
\end{lemma}

The map $\phi\colon T\to\bC$, $z\mapsto z-g(z)$, is well defined by~\eqref{3}.
Let $J_\phi=1-|\overline{\partial} g|^2$ be the Jacobian determinant of $\phi$.
Then $D^+=\{z\in T\colon J_\phi(z)>0\}$ is the set where  $\phi$ preserves the orientation and
$D^-=\{z\in T\colon J_\phi(z)<0\}$ is the set where  $\phi$ reverses the orientation.
As $\phi$ has one pole, 
a point $w$ of large modulus has one preimage on $T$ and the map is reversing
orientation at this preimage.
We conclude that the degree of $\phi$ equals $-1$; see~\cite[\S 5]{Milnor} for the definition 
of the degree.

Suppose first that all zeros of $\phi$ are in $D^+\cup D^-$.
Equivalently, $g$ has no neutral fixed points.
Denote by $N^+$ and $N^-$ the numbers of zeros of $\phi$  in
$D^+$ and $D^-$ respectively. Then $N^+-N^-$ equals the degree of $\phi$ so that
$N^+-N^-=-1$ and thus 
\begin{equation}\label{count}
N^-=N^++1.
\end{equation}
For the number $N=N^++N^-$ of fixed points of $g$ in $T$ we thus find that 
\begin{equation}\label{count2}
N=2N^++1.
\end{equation}

Since $J_\phi=1-|\overline{\partial} g|^2$, the zeros of $\phi$ in $D^+$ are attracting 
fixed points of $g$ while the zeros of $\phi$ in $D^-$ are repelling fixed points of~$g$.
Thus Lemma~\ref{lemma4} yields that
\begin{equation}\label{count3}
N^+\leq 2.
\end{equation}
It follows from~\eqref{count2} and~\eqref{count3} that  $N\leq 5$.

On the other hand, since $\phi$ is odd and $\Lambda$-periodic it easily follows that 
that the half-periods $\omega_1/2$, $\omega_2/2$ and $\omega_3/2=(\omega_1+\omega_2)/2$ are 
zeros of~$\phi$. Equivalently, they are fixed points of~$g$. Thus we have $N\geq 3$.
Altogether, since $N$ is odd by~\eqref{count2}, it follows that $N=3$ or $N=5$.

It remains to determine the criterion distinguishing the cases.
Suppose first that all three half-periods are in $D^-$; that is,
they are repelling fixed points of~$g$. Then $N^-\geq 3$.
This yields that $N^-=3$ and $N^+=2$ so that $N=5$ by~\eqref{count} and~\eqref{count3}.
Suppose now that one half-period, say $\omega_j/2$,
is not in $D^-$ and thus 
in $D^+$. Thus $\omega_j/2$ is an attracting fixed point of $g$ and
hence attracts a critical orbit by Fatou's theorem.
However, since $g$ is odd we see that $\omega_j/2$ in fact attracts
both critical orbits. Thus, by Fatou's theorem, there are no other 
attracting 
fixed points. Thus $N^+=1$ and hence $N=3$ in this case.

We see that $N=3$ if and only if there exists $j\in\{1,2,3\}$ such that
$\omega_j/2$ is an attracting fixed point of $g$; that is,
$|\overline{\partial} g(\omega_j/2)|< 1$.
Using~\eqref{dg} this takes the form 
\begin{equation}\label{e_jab}
\min_{1\leq j\leq 3} \left|\frac{e_j}{b}-\frac{a}{b}\right|< 1.
\end{equation}
Now~\eqref{ab} yields 
\[
\frac{a}{b}
=-\frac{\overline{\omega_1}}{\omega_1}-\frac{\eta_1}{b\omega_1}
=-\frac{\overline{\omega_1}}{\omega_1}+\frac{\eta_1|\omega_1|^2\im\tau}{\omega_1\pi}
=\frac{\overline{\omega_1}}{\omega_1}\left(-1+\frac{\eta_1\omega_1\im\tau}{\pi}\right)
\]
and
\[
\frac{e_j}{b}
=-\frac{e_j|\omega_1|^2\im\tau}{\pi}
=-\frac{\overline{\omega_1}}{\omega_1}
\frac{e_j\omega_1^2\im\tau}{\pi}.
\]
Substituting the last two equations in~\eqref{e_jab} yields
\[
\left|\frac{e_j\omega_1^2+\eta_1\omega_1}{\pi}\im\tau-1\right|< 1.
\]
This is equivalent to strict inequality in~\eqref{criterion2}.
This completes the proof of the theorem in the case that all zeros of $\phi$ are 
in $D^+\cup D^-$. 

To deal with the case where this condition is not satisfied, we note that 
in the above arguments we may replace $\phi(z)$ by $\psi(z)=\phi(z)-w$
for any $w\in\C$. Noting that 
$\phi$ and $\psi$ have the same Jacobian determinant
we conclude that 
whenever all $w$-points of $\phi$ are in $D^+\cup D^-$,
then $\phi$ has either three or five $w$-points in~$T$.
Moreover, there are three $w$-points if and only if one them is in~$D^+$.

We will use the following result~\cite[Proposition~3]{BE2010}.

\begin{lemma} \label{lemma5}
Let  $D\subset \C$ be a domain and let $f\colon D\to\C$ be a harmonic map.
Suppose that there exists $m\in\N$ such that every $w\in\C$ has at
most $m$ preimages.
Then the set of points which have $m$ preimages is open.
\end{lemma}

Suppose now that $\phi$ has a zero in $T\backslash (D^+\cup D^-)$.
Then there are arbitrarily small $w$ such that $\phi$ has a $w$-point in~$D^+$.
Since for such $w$ the function $\phi$ has three $w$-points, it follows
from Lemma~\ref{lemma5} that $\phi$ has three zeros in~$T$.
These zeros are the half-periods and we see that we have equality in~\eqref{criterion2}.
This completes the proof of the theorem.

\begin{remark} 
The quantities $e_j\omega_1^2+\eta_1\omega_1$ occurring in~\eqref{criterion2} have the 
following representations
via theta functions, see~\cite[p.~44]{Weierstrass}:
\[
e_1\omega_1^2+\eta_1\omega_1
=- \frac{\vartheta_2''(0)}{\vartheta_2(0)}
, \quad 
e_2\omega_1^2+\eta_1\omega_1
=- \frac{\vartheta_0''(0)}{\vartheta_0(0)}
, \quad 
e_3\omega_1^2+\eta_1\omega_1
=- \frac{\vartheta_3''(0)}{\vartheta_3(0)}.
\]
The series for theta functions converge fast and provide a convenient way to compute Figure~\ref{fig1}.

Let 
\[
F_j(\tau)=
\frac{2\pi i}{e_j\omega_1^2+\eta_1\omega_1} - \tau
\]
so that~\eqref{criterion2} takes the form $\im F_j(\tau)\geq 0$.
It seems that $F_j$ maps the components of the set of all $\tau$ in the upper half-plane
where $\im F_j(\tau)> 0$
univalently onto the upper half-plane, but we have not been able to prove this.
\end{remark}

\begin{remark} 
Lin and Wang~\cite[Theorems 1.6 and 1.7]{Lin} pay special attention to the case 
that $\tau=1/2+ib$, in which case Green's function has five critical 
points if and only if $b$ is outside a certain interval $[b_0,b_1]$ where
$b_0\approx 0.35$ and $b_1\approx 0.71$.

Here we note that the constants $b_0$ and $b_1$ are related to the 
so-called one-ninth constant~$\Lambda$ occurring in approximation theory;
see~\cite[Section 4.5]{Finch} for the definition and properties of this constant.
As noted there,
this constant was already computed by Halphen~\cite[p.~287]{Halphen} to six digits.
It turns out that $\Lambda=e^{-2\pi b_0}=e^{-\pi/(2 b_1)}$. The numerical values are
$\Lambda=0.10765391\dots$, $b_0=0.35472989\dots$ and $b_1=0.70476158\dots$.

To prove the above relation between these constants we note first that
$e_1\omega_1^2+\eta_1\omega_1=0$ if and only if $\vartheta_2''(0)=0$
and thus if
\[
\sum_{k=0}^\infty (2k+1)^2 h^{k(k+1)}= 0 \quad \text{where}\   h=e^{i\pi\tau}.
\]
For $\tau=1/2+ib$ with $b\in\R$ we have $h=e^{i\pi/2-\pi b}$ and hence
\[
h^{k(k+1)}=\exp\!\left( (i\pi-2\pi b)\frac{k(k+1)}{2}\right) =(-x)^{k(k+1)/2}
\quad \text{with} \ x=e^{-2\pi b}.
\]
For $\tau=1/2+ib$ the condition that  $e_1\omega_1^2+\eta_1\omega_1=0$
is thus equivalent to
\[
\sum_{k=0}^\infty (2k+1)^2 (-x)^{k(k+1)/2}=0.
\]
The smallest positive solution of the last equation is the
one-ninth constant~$\Lambda$; see~\cite[Section 4.5]{Finch}.
Since $x=e^{-2\pi b}$ we deduce that $\Lambda=e^{-2\pi b_0}$.

The modular group leaves the sets where Green's function has five or three
critical points invariant. The transformation $T(z)=(z-1)/(2z-1)$
satisfies $T(1/2+ib)=1/2+i/(4b)$.
This implies that $b_1=1/(4b_0)$ and $\Lambda=e^{-\pi/(2 b_1)}$.
\end{remark}
\begin{remark}
Figure~\ref{fig2} shows the Julia sets of the functions  $g$ corresponding to the square and the
hexagonal lattice. 
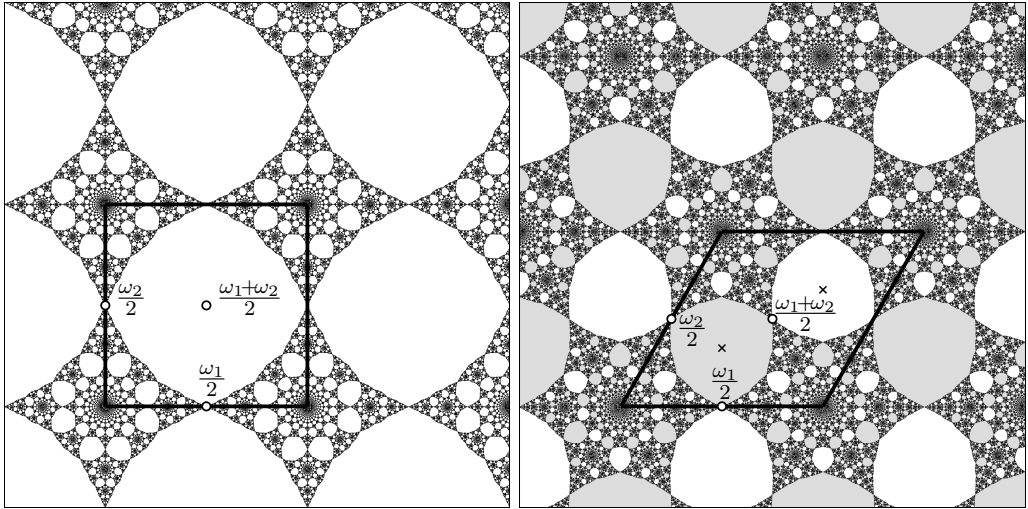
\begin{figure}[htb]
\begin{center}
\begin{overpic}[width=0.49\textwidth]{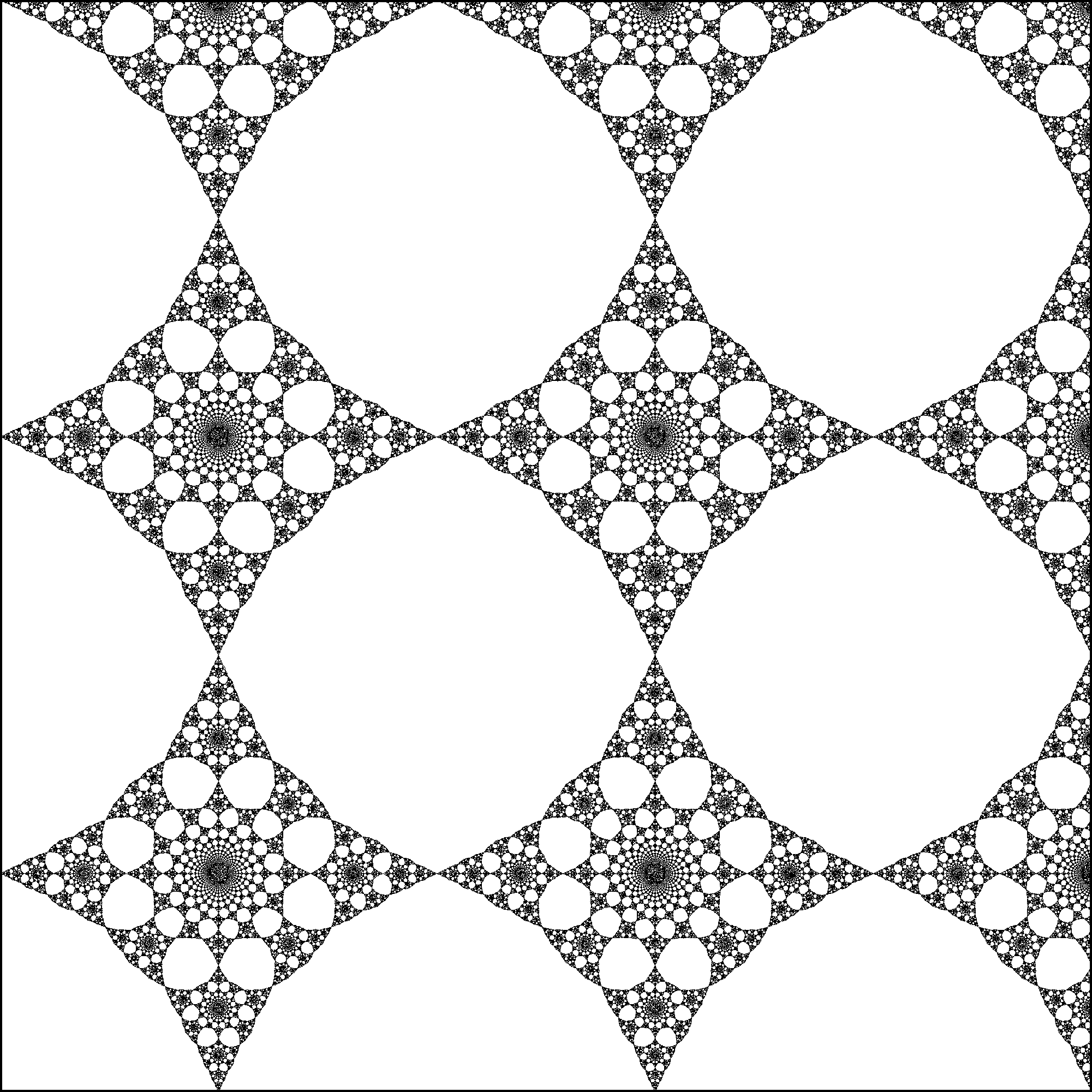}
\linethickness{0.5mm}
 \put (20,20) {\line(1,0){40}}
 \put (20,60) {\line(1,0){40}}
 \put (20,20) {\line(0,1){40}}
 \put (60,20) {\line(0,1){40}}
 \put(40,20){\circle*{2}}
 \put(20,40){\circle*{2}}
 \put(40,40){\circle*{2}}
 \put(40,20){\textcolor{white}{\circle*{1.3}}}
 \put(20,40){\textcolor{white}{\circle*{1.3}}}
 \put(40,40){\textcolor{white}{\circle*{1.3}}}
 \put(22,40){{\small $\frac{\omega_2}{2}$}}
 \put(42,40){{\small $\frac{\omega_1\!+\!\omega_2}{2}$}}
 \put(38,24){{\small $\frac{\omega_1}{2}$}}
\end{overpic}
\begin{overpic}[width=0.49\textwidth]{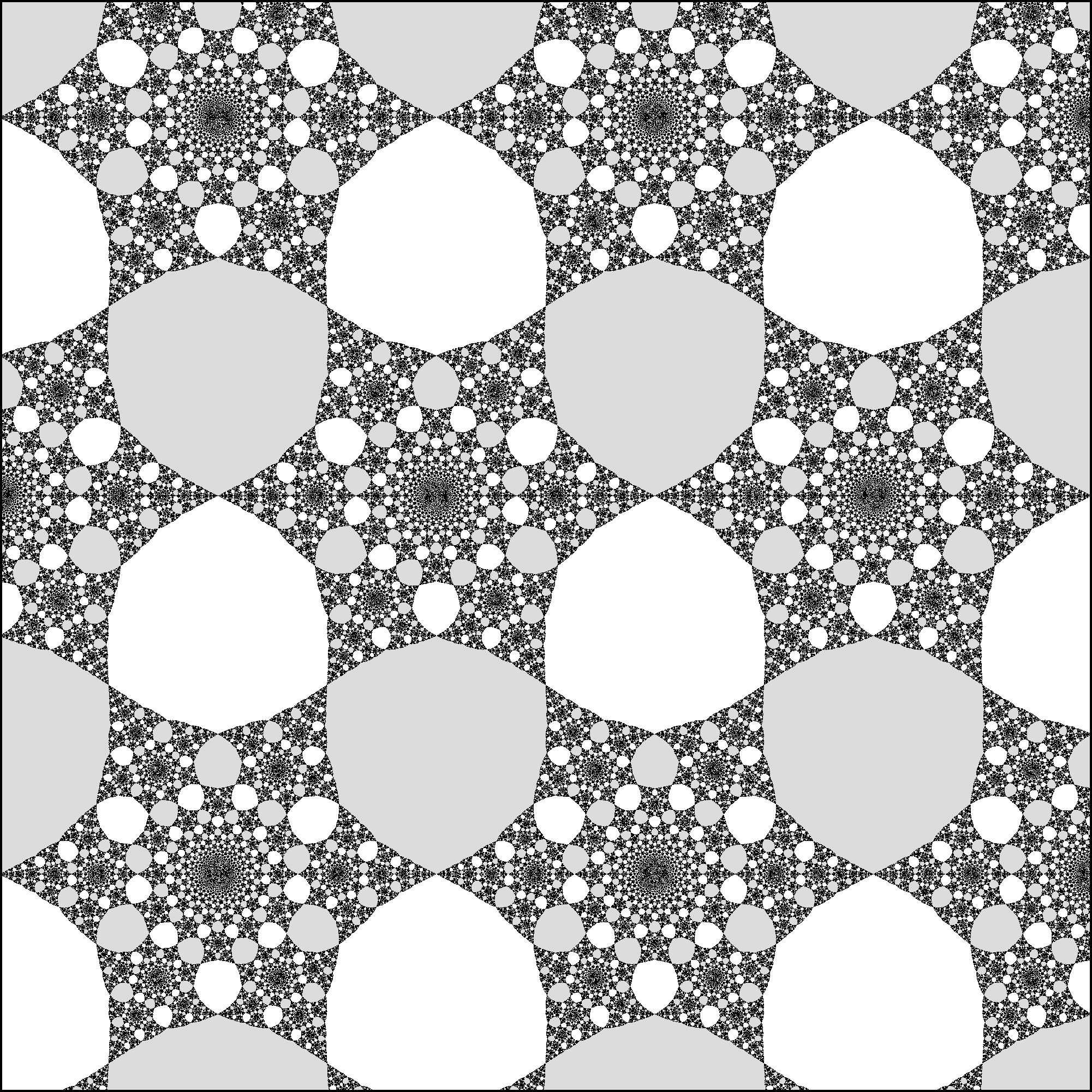}
\linethickness{0.5mm}
 \put (20,20) {\line(1,0){40}}
 \put (40,54.64) {\line(1,0){40}}
 \put (20,20) {\line(1,1.732){20}}
 \put (60,20) {\line(1,1.732){20}}
 \put(40,20){\circle*{2}}
 \put(30,37.32){\circle*{2}}
 \put(50,37.32){\circle*{2}}
\linethickness{0.2mm}
 \put(40,31.547){\line(1,1){0.7}}
 \put(40,31.547){\line(1,-1){0.7}}
 \put(40,31.547){\line(-1,1){0.7}}
 \put(40,31.547){\line(-1,-1){0.7}}
 \put(60,43.094){\line(1,1){0.7}}
 \put(60,43.094){\line(1,-1){0.7}}
 \put(60,43.094){\line(-1,1){0.7}}
 \put(60,43.094){\line(-1,-1){0.7}}
 \put(40,20){\textcolor{white}{\circle*{1.3}}}
 \put(30,37.32){\textcolor{white}{\circle*{1.3}}}
 \put(50,37.32){\textcolor{white}{\circle*{1.3}}}
 \put(31,34){{\small $\frac{\omega_2}{2}$}}
 \put(50,37){{\small $\frac{\omega_1\!+\!\omega_2}{2}$}}
 \put(37.5,23.5){{\small $\frac{\omega_1}{2}$}}
\end{overpic}
\caption{Julia sets corresponding to $\tau=i$ and $\tau=e^{i\pi/3}$.}
\label{fig2}
\end{center}
\end{figure}
The standard fundamental domains are marked by thick lines and the
half-periods are marked by circles. For the square lattice, the
half-periods are the only
fixed points of~$g$, with $(\omega_1+\omega_2)/2$ attracting and the
two other ones repelling.
For the hexagonal lattice, all half-periods are repelling and there
are the attracting fixed points $(\omega_1+\omega_2)/3$ and $2(\omega_1+\omega_2)/3$
marked by crosses.

Figure~\ref{fig3} shows the Julia set of the function corresponding 
 to $\tau=1/2+ib_1$. The half-period $\omega_1/2$ is a parabolic fixed point with two petals.
The two other half-periods are repelling fixed points.
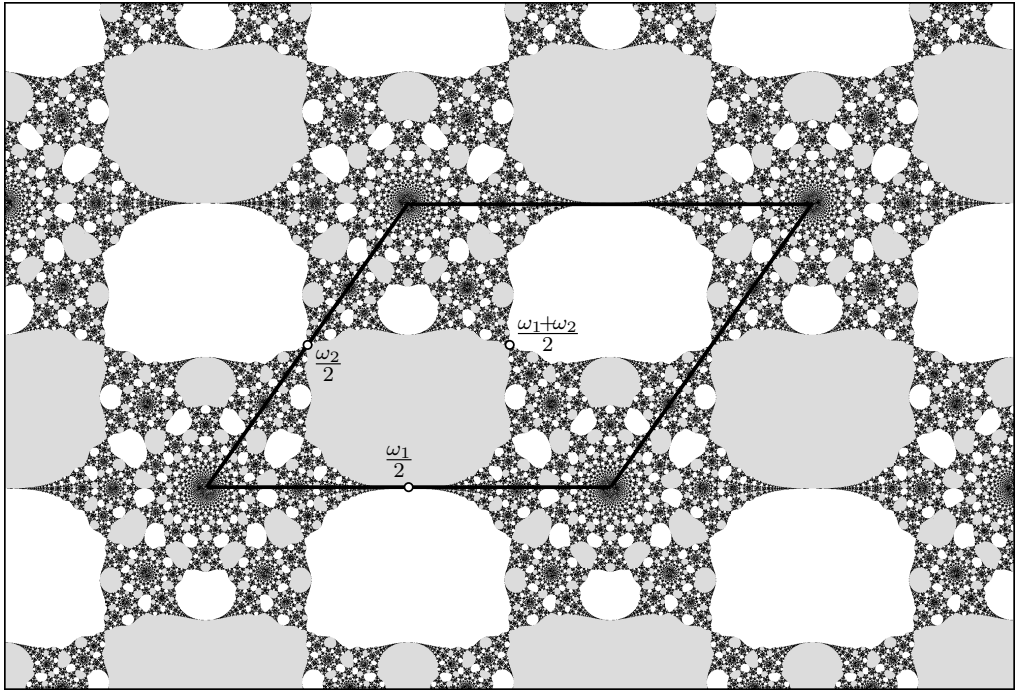
\begin{figure}[htb]
\begin{center}
\begin{overpic}[width=0.98\textwidth]{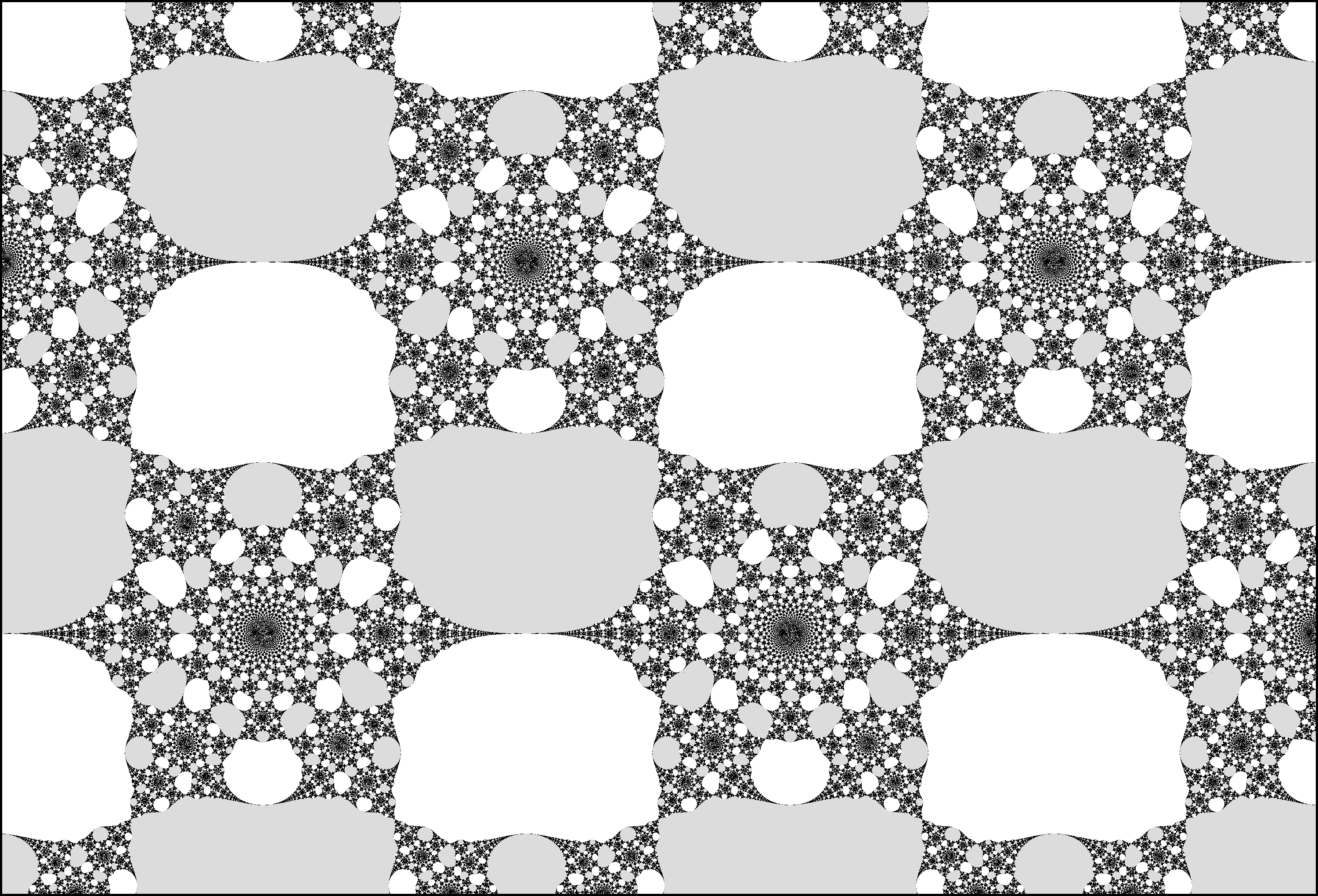}
\linethickness{0.5mm}
 \put (20,20) {\line(1,0){40}}
 \put (40,48.02) {\line(1,0){40}}
 \put (20,20) {\line(1,1.409){20}}
 \put (60,20) {\line(1,1.409){20}}
 \put(40,20){\circle*{1}}
 \put(30,34.09){\circle*{1}}
 \put(50,34.09){\circle*{1}}
 \put(40,20){\textcolor{white}{\circle*{0.65}}}
 \put(30,34.09){\textcolor{white}{\circle*{0.65}}}
 \put(50,34.09){\textcolor{white}{\circle*{0.65}}}
 \put(30.5,31.5){{\small $\frac{\omega_2}{2}$}}
 \put(50.5,34.5){{\small $\frac{\omega_1\!+\!\omega_2}{2}$}}
 \put(37.5,22.0){{\small $\frac{\omega_1}{2}$}}
\end{overpic}
\caption{Julia set corresponding to $\tau=1/2+ib_1$.}
\label{fig3}
\end{center}
\end{figure}
\end{remark}

\begin{comment}
Besides the application in \cite{KS,KN1,KN2,KL}, anti-holomorphic
dynamics were studied for its own sake in \cite{Rip,Schl1,Schl} and elsewhere.
In these papers, the iteration of the anti-holomorphic map
$\overline{z}^2+c$ was investigated. The bifurcation diagram in the $c$-plane
for this map is called the Mandelbar set or Tricorn. It was noticed that
the neutral cycle appears on a whole arc of the boundary of the hyperbolic
component. 

In the case we considered, a striking new phenomenon occurs: the neutral cycle
occurs {\em everywhere} on the boundary of the hyperbolic component.

To be more precise, let $H$ be the upper half-plane, so $\tau\in H$.
Let $X$ be the maximal open set of $\tau$ where $G$ has three critical points,
and $Y=H\backslash\overline{X}$. For $\tau$ in~$X$, the critical orbits of
$g=g_\tau$ tend to an attracting fixed point, while for $\tau\in Y$ they tend
to two attracting fixed points. The open sets $X$ and $Y$ and their common boundary
which is a piecewise analytic curve exhaust the whole parameter space!
So a generic $g_\tau$ in our family has hyperbolic dynamics in a trivial way.

Such a situation occurs in holomorphic dynamics only in trivial cases~\cite[Theorem~2.2]{McMullen1987},
and was not encountered so far in anti-holomorphic dynamics.

The point in a hyperbolic component at which the multiplier is zero is called
the center of this hyperbolic component. Computation shows that the centers
of $X$ and $Y$ in the standard fundamental domain
are $\tau=i$ and $\tau=e^{2\pi i/3}$, corresponding to the
square and hexagonal lattices respectively. 
\end{comment}

We thank C.-S. Lin for his useful comments on this paper.

{\em Mathematisches Seminar

Christian-Albrechts-Universiat zu Kiel

Ludewig-Meyn-Str. 4

24098 Kiel

Germany
\vspace{.1in}

Purdue University

West Lafayette, IN 47907

USA}
\end{document}